\documentclass[a4paper,12pt]{article}
 
\usepackage{amsmath,amsthm,amssymb}
\usepackage{graphicx}
\usepackage{dsfont}
\usepackage{enumerate}
\usepackage{mathrsfs}
\usepackage[colorlinks]{hyperref}

\usepackage{constants}
\newconstantfamily{global}{
  symbol=c, 
}
\newcommand{\SymbolDist}{\mathfrak{d}}
\newconstantfamily{dist}{
  symbol=\SymbolDist,
}
\newconstantfamily{abs}{
  symbol=\bd{c}, 
}

\newtheorem{theorem}{Theorem}[section]
\newtheorem{lemma}[theorem]{Lemma}

{\theoremstyle{definition}
\newtheorem{definition}[theorem]{Definition}
}

\newcommand{\R}{\mathds{R}}
\newcommand{\N}{\mathds{N}}
\newcommand{\Sp}{\mathds{S}}
\newcommand{\B}{B^d}

\newcommand{\bd}{\boldsymbol}
\newcommand{\dint}{\mathrm{d}}
\newcommand{\origin}{\boldsymbol{o}}
\newcommand{\MetricBoundary}{d_m}
\newcommand{\MeanWidth}{b}
\newcommand{\FunctionBound}[1][l]{\mathfrak{b}_{#1}}
\newcommand{\ShapeFactor}[1][l]{\mathfrak{g}_{#1}}
\newcommand{\ShapeFactorBis}[1][l]{\mathfrak{f}_{#1}}
\renewcommand{\Cap}[3]{\mathrm{cap} \left( #1 , #2 , #3 \right) }
\newcommand{\vect}[1]{\boldsymbol{ #1 }}
\newcommand{\card}[1]{|#1|}
\newcommand{\Approx}{\boldsymbol{\Theta}}
\renewcommand{\mid}{:}

\begin{document}

\title{Polytopal Approximation of Elongated Convex Bodies}
\author{Gilles Bonnet \thanks{gilles.bonnet@rub.de} }
\maketitle

  Primary
  52A27,  
  52B11;  
  Secondary
  52B60.  

\abstract{This paper presents bounds for the best approximation, with respect to the Hausdorff metric, of a convex body $K$ by a circumscribed polytope $P$ with a given number of facets.
  These bounds are of particular interest if $K$ is elongated.
  To measure the elongation of the convex set, its isoperimetric ratio $ V_j(K)^{1/j} V_i(K)^{-1/i} $  is used.
    }

\section{Introduction and main results}
Let $ K \subset \R^d $ be a full dimensional convex body, i.e. a compact convex set with non-empty interior. 
Fix $n \in \N$ and denote by $ P = P_{K,n} \supset K$ a circumscribed polytope, with at most $n$ facets,  minimizing the Hausdorff distance $ d_H (K,P) $.
It is well known that $d_H(K,P)$ is of order $n^{- 2/(d-1)}$ (see, e.g., \cite{Gruber1993book}).  
Thus, we set $ \Cl[dist]{boundC} ( K , n ) := d_H (K,P) n^{2/(d-1)}$ and $ \Cl[dist]{boundCbis} ( K , n ) := \sup_{ m \geq n } \Cr{boundC} ( K , m ) $,
and we have
\[ d_H (K,P) 
= \frac{ \Cr{boundC} ( K , n ) }{n^{2/(d-1)}}
\leq  \frac{ \Cr{boundCbis} ( K , n ) }{n^{2/(d-1)}}.
\] 
Estimating $\Cr{boundC} ( K , n )$ and $\Cr{boundCbis} ( K , n )$ is a classical problem.
We refer the reader to the well known surveys of P. M. Gruber \cite{Gruber1993book,Gruber1994} and E. M. Bronstein \cite{Bronstein2008} for an excellent overview of the huge amount of results and literature about polytopal approximation.
The specificity of our main result is that we take into account how much $K$ is `\textit{elongated}'.

We denote by $V_i$ the $i$-th intrinsic volume (see Section \ref{sec:Setting} for the definition). 
For any $ 1 \leq i < j \leq d $, we call $ V_j(K)^{1/j} V_i(K)^{-1/i} $ the \textit{$(i,j)$-isoperimetric ratio} of $K$.
It is scale and translation invariant.
The isoperimetric inequality (see Section \ref{sec:Setting}, inequality \eqref{eq:IsoperimetricIneq}, for a statement) says that it is maximized by the balls.
On the other hand, $ V_j(K)^{1/j} V_i(K)^{-1/i} \simeq 0 $ precisely when the normalized body $ V_i(K)^{-1/i} K $ is close to a $(j-1)$-dimensional convex body.
If an isoperimetric ratio of $K$ is close to zero, we say that $K$ is \textit{elongated}.
More precisely, if $ V_j(K)^{1/j} V_i(K)^{-1/i} < \epsilon $, we say that 
\textit{$K$ is $ ( \epsilon \colon i , j ) $-elongated}.
The following theorem gives a bound for the Hausdorff distance between a convex body $K$ and its best approximating polytope.
This bound can be arbitrarily small if $K$ is sufficiently elongated.

\begin{theorem}
  \label{thm:main}
  Assume $ 1 \leq i < j \leq \lceil (d-1) /2 \rceil $.
  Set $\alpha=2 \lceil (d-1)/2 \rceil (d-1) d^{-1}$
  and $ \beta = \lceil (d-1)/2 \rceil (d-1)^{-1} d^{-1} $.
  There exist constants $\delta_{i,j}$ and $n_{i,j}$, both depending on $d$, such that the following holds.
  For any $ \epsilon >0 $ and any convex body $ K $
  \[ \text{ if } K \text{ is } (\epsilon\colon  i,j) \text{-elongated then } \Cr{boundCbis} \left( K , n_{i,j} \epsilon^{-\alpha} \right) 
  < \delta_{i,j} \epsilon^\beta V_1(K) . \]
  I.e. for any $ \epsilon>0 $ and any convex body $K$
  \[ \text{ if } \frac { V_j(K)^{1/j} } { V_i(K) ^{1/i} } < \epsilon 
  \text{ then } d_H ( K , P ) < \delta_{i,j} \epsilon^\beta \frac{V_1(K)}{n^{2/(d-1)}} \text{ for any } n \geq n_{i,j} \epsilon^{-\alpha} ,\]
  where $ P = P_{K,n} \supset K$ is a circumscribed polytope, with at most $n$ facets,  minimizing the Hausdorff distance $ d_H (K,P) $.
\end{theorem}
Note that the case $i=1$ and $ j = \lceil (d-1) /2 \rceil $ implies all the others.
This is a consequence of the isoperimetric inequality.
We conjecture that Theorem \ref{thm:main} remains true for any $  1 \leq i < j \leq d-1 $
and that $\beta$ could be replaced by $1$.
Equation $\eqref{eq:BnAsymptotic}$ below gives support to this conjecture.

\medskip
Let us recall a few important results in order to motivate Theorem \ref{thm:main}.
If $K$ has a twice differentiable smooth boundary, we have a precise asymptotic approximation of $\Cr{boundCbis} ( K , n )$. After planar results due to T\'oth \cite{Toth1948} and McClure and Vitale \cite{McClure1975},   Schneider \cite{Schneider1981,Schneider1987} and Gruber \cite{Gruber1993article} succeeded in proving that
\begin{equation*}
\lim_{n \to \infty }\Cr{boundCbis} (K,n)
= \frac12 \left( \frac{\vartheta_{d-1}}{\kappa_{d-1}} \int_{\partial K} \kappa_C (\vect{x})^{1/2} \sigma ( \dint \vect{x}) \right) ^{2/(d-1)}, 
\end{equation*}
where $\vartheta_k$ is the minimum covering density of $\R^k$ with balls of fixed radius, 
$\kappa_k$ the volume of the $k$-dimensional ball,
$\kappa_C(\vect{x})>0$ the Gaussian curvature of $K$ at the point $\vect{x}$,
and $\sigma(\cdot)$ the surface area measure.
More recently, B\"or\"oczky \cite{Boroczky2000} removed the condition $\kappa_C(\vect{x})>0$. 

In many practical situations it is out of reach to compute the integral explicitly.
But if $K$ is elongated, we can have a good upper bound.
H\"older's inequality implies
\begin{align*}
\int_{\partial K} \kappa_C (\vect{x})^{1/2} \dint \sigma (\vect{x})
& \leq (2 d \kappa_d)^{1/2} V_{d-1} (K) ^{1/2}
\\
&= (2 d \kappa_d)^{1/2} \left[ \frac{ V_{d-1} (K) ^{1/(d-1)} }{ V_1 (K) } \right]^{(d-1)/2} V_1(K) ^{(d-1)/2}.
\end{align*}
Hence, with the isoperimetric inequality, for any $ 1 \leq i < j \leq d-1 $, we have that
\begin{equation}
\label{eq:BnAsymptotic}
\text{ if } K \text{ is } (\epsilon\colon  i,j) \text{-elongated then } 
\lim_{n \to \infty }\Cr{boundCbis} (K,n)
\leq \delta'_{i,j} \epsilon V_1(K) 
\end{equation}
with
\[ \delta'_{i,j} := 
\frac12 \left( \frac{\vartheta_{d-1}}{\kappa_{d-1}} \right)^{2/{d-1}} 
( 2 d \kappa_d )^{1/(d-1)} 
\frac{ V_{d-1} (\B) ^{1/(d-1)} }{ V_1 (\B) }
\frac{ V_i (\B) ^{1/i} }{ V_j (\B) ^{1/j} }.
\]
Therefore, we have a good asymptotic bound for elongated smooth convex bodies.

\medskip
The main goal of this paper is to extend these results to the non-asymptotic and non-smooth case.
The order $\epsilon$ in $\eqref{eq:BnAsymptotic}$ should be compared to the order $\epsilon^\beta$ in Theorem \ref{thm:main} for a fixed $n$.
It is especially of interest, for example, if we approximate a polytope with many facets by one with fewer facets.
This was considered by Reisner, Sch{\"u}tt and Werner in \cite{ReisnerSchuttWerner01}.
This paper was the starting point of our investigations.
A reader who has studied it will notice that principal ideas of their work are still present in our proofs.

Theorem \ref{thm:main} should be compared to the following result.
It was obtained independently in \cite{Bronshteyn1975} and \cite{Dudley1974}.
The constants were improved in \cite{ReisnerSchuttWerner01}.
There exist constants
$ \Cl[global]{boundFixedR} ( d ) $ 
and $ \Cl[global]{boundN}(d) $ 
such that 
$ \Cr{boundCbis} ( K , \Cr{boundN}(d) ) < \Cr{boundFixedR} ( d ) R(K) $,
i.e.
\begin{equation*}
d_H (K,P) \leq \Cr{boundFixedR} (d) \frac{ R(K) }{n^{2/(d-1)}} \text{ for } n > \Cr{boundN}(d) ,
\end{equation*}
where $ R(K) $ is the radius of the smallest ball containing $K$.
Note that $ R(K) $ is of the same order as $V_1(K)$.
Although this bound is sharp in general, it is worse if we assume that $K$ is elongated.
The following is an example of such a situation.
Fix a small $ \epsilon > 0 $.
Let $ K \subset \R^4 $ be a convex body.
It is well known that there exists an ellipsoid $ E $ such that
$ E \subset K \subset d E $ (see, e.g., \cite{John1948}).
Let $ r_1 > \cdots > r_4 $ be the lengths of the principal axes of $E$.
Assume that $ r_1 = 1 $ and $ r_2 $ is sufficiently small.
Thus, $K$ is $(\epsilon \colon 1,2)$-elongated.
For $ n > n_{1,2} \epsilon^{-3} $, Theorem~\ref{thm:main} says that 
$ \Cr{boundC} ( K , n ) 
< \delta_{1,2} \epsilon^{1/6}
\ll \Cr{boundFixedR} (d) R(K) $.

\medskip
Finally, we would like to highlight the following theorem.
Not only is it an important step in the proof of Theorem \ref{thm:main} but also an interesting result on its own.

\begin{theorem}
  \label{thm:intermediate}
  There exist absolute constants $\Cr{abs:1}$ and $\Cr{abs:2}$, independent of $d$, such that the following holds.
  Let $K$ be a convex body.
  Then, 
  $$ \Cr{boundCbis} ( K , \Cr{abs:1}^d d^{d/2}
  V_{d-1}(K+\B ) ) 
  < \Cr{abs:2} d 
  V_{d-1}(K+\B )^{2/(d-1)} .$$
  I.e. for any integer $ n > \Cr{abs:1}^d d^{d/2} V_{d-1}(K+\B ) $,
  there exists a polytope
  $P\supset K$ with $n$ facets such that
  $$\mathrm{d}_H(K,P) < \Cr{abs:2} d V_{d-1}(K+\B )^{2/(d-1)} n^{-2/(d-1)}.$$
\end{theorem}

\medskip
The paper is structured as follows.
In the next section setting, notation and background material from convex geometry are provided.
In Section \ref{sec:deltaNet} we give a general framework to build a $\delta$-net on an abstract measured metric space satisfying mild properties, and we apply it to prove Theorem \ref{thm:intermediate}.
The proof of Theorem \ref{thm:main} is given in Section \ref{sec:MainProof}. 
It uses a shape factor introduced and described in Section \ref{sec:ShapeFactor}.

\section{Setting, notation and background}
\label{sec:Setting}
We work in the euclidean space $\R^d$ with origin $\origin$, scalar product $ \langle \cdot , \cdot \rangle $ and associated norm $|\cdot|$.
We denote by $ B ( \vect{x} , r ) $ and $ S ( \vect{x} , r ) $, respectively, the ball and the sphere of center $ \vect{x} $ and radius $r$.
The unit ball $ \B = B ( \origin , 1 )$ has volume $\kappa_d$
and the unit sphere $ \Sp^{d-1} = S ( \origin , 1 ) $ has surface area $\omega_d = d \kappa_d $.
We denote by $\mathcal{K}$ the set of \textit{convex bodies} (compact and convex sets of $\R^d$) with at least $2$ points.
This set is equipped with the Hausdorff distance
\[
 d_H (K,L) = \min_{r\geq0} \left( K \subset L + r \B , L \subset K + r \B \right) 
\] and its associated topology and Borel structure.
The same holds for the future subsets of $\mathcal{K}$ that we will encounter in this paper.
The set $\mathcal{K}$ is also equipped with Minkowski sum and scale action. 
For any $ t \in \R $ and $ A , B \in \mathcal{K} $, we have
\[
t A := \{ t \vect{a} \mid \vect{a} \in A \},
\quad
A + B := \{ \vect{a} + \vect{b} \mid \vect{a} \in A , \vect{b} \in B \}.
\]
We denote by $ \partial K $ the boundary of a given convex body $ K $.

Let $f:\mathcal{K}\to\R$ be a map.
If there exists $k\in\R$ such that 
$ f( t K ) = t^k K $ for any $ K \in\mathcal{K}$ and $t>0$,
we say that $f$ is \textit{homogeneous of degree $k$}.
We say that $f$ is \textit{scale invariant} if $f$ is homogeneous of degree $0$.
If $f(K+\vect{x})=f(K)$ for any $K\in\mathcal{K}$ and $\vect{x}\in\R^d$, we say that $f$ is \textit{translation invariant}.
We say that $f$ is a \textit{shape factor} if $f$ is scale and translation invariant.

For the following facts of convex geometry we refer the reader to \cite{Gruber07}.

\textbf{The Steiner Formula and Intrinsic Volumes.}
We denote by $V_d(\cdot)$ the volume, i.e., the $d$-dimensional Lebesgue measure.
The Steiner Formula tells us that there exist functionals 
$V_i : \mathcal{K} \to [0,\infty) $, for $ 0 \leq i \leq d $, such that
for any $K\in \mathcal{K}$ and $ \epsilon \geq 0 $
\[
V_d ( K + \epsilon \B ) = \sum_{i=0}^d \epsilon^{d-j} \kappa_{d-j} V_j(K).
\]
$V_i(K)$ is called the \textit{$i$-th intrinsic volume} of $K$.
Some of the intrinsic volumes have a clear geometric meaning.
$V_d$ is the volume.
If $K$ has non-empty interior, then
\[ V_{d-1} (K) = \frac12 \mathcal{H}^{d-1} ( \partial K ), \]
where $ \mathcal{H}^{d-1} ( \partial K ) $ is the $ (d-1) $-dimensional Hausdorff measure of the boundary $ \partial K $ of $ K $.
Thus, $ 2 V_{d-1} $ is the surface area.
$V_1$ is proportional to the \textit{mean width $\MeanWidth$}.
More precisely,
\[ 
\frac{d \kappa_d}{2} \MeanWidth (K) 
= \kappa_{d-1} V_1 (K)
= \int_{\Sp^{d-1}} h ( K , \vect{u} ) \sigma( \dint \vect{u} ),
\]
where $\sigma$ is the surface area measure on $\Sp^{d-1}$
and $h(K, \vect{u} ):=\max\{\langle \vect{x}, \vect{u} \rangle\mid \vect{x}\in K\}$ is the value of the \textit{support function} of $K$ at $ \vect{u} $.
$V_0(K)=1$ is the Euler characteristic.
For $ 1 \leq i < j \leq d $ and $K\in\mathcal{K}$, we call the shape factor $ V_j(K)^{1/j} V_i(K)^{-1/i} $ the $(i,j)$-\textit{isoperimetric ratio} of $K$.

\textbf{The Isoperimetric Inequality.}
Let $B \subset \R^d$ be a $d$-dimensional ball.
For any $ K \in \mathcal{K} $ and for any $ 1 \leq i < j \leq d $,
\begin{equation}
\label{eq:IsoperimetricIneq}
V_j(K) ^{1/j} \leq \frac{ V_j ( B ) ^{1/j} }{ V_i ( B ) ^{1/i} } V_i(K) ^{1/i},
\end{equation}
with equality if and only if $K$ is a ball.

\textbf{A Steiner-type Formula.}
For any $ K \in \mathcal{K} $
\begin{equation}
\label{eq:SteinerType}
V_{d-1} (K+\B) 
= \sum_{k=0}^{d-1} \frac{(d-k) \kappa_{d-k} }{2d} V_k (K).
\end{equation}

The isoperimetric inequality and the Steiner-type formula imply the next fact.

\textbf{Fact :}
Let $ d \geq 3 $,
$I$ be an interval (convex hull of two distinct points),
$B$ be a ball,
and $ K \in \mathcal{K} $ neither an interval nor a ball.
Assume that $ V_1(I) = V_1(K) = V_1(B) $.
Note that $ V_1(I) $ is just the length of the segment $I$.
Then, we have
\begin{equation}
\label{eq:IneqSurfaceConvexPlusBall}
V_{d-1} ( I + \B ) 
< V_{d-1} ( K + \B )
< V_{d-1} ( B + \B ) .
\end{equation}

\textbf{Convention about the constants.}
The constants are denoted by $c_i$, where $i$ is an index.
They depend on $d$ but are independent of any other quantity.
We estimate the dependence on $d$ using the Landau notation.
By $c_i = \Approx ( f(d) ) $ we mean that there exist absolute constants $ k_0 , k_1 > 0 $ such that
$ k_0 f(d) < c_i < k_1 f(d) $
for any $d$.

\section{\texorpdfstring{$\delta$}{delta}-net and polytopal approximation}
\label{sec:deltaNet}
First, let us set some notation.
Assume $ M $ is metric space with distance $ d_M $.
I.e. a set $ M $ and a function $ d_M \colon M \times M \to [0,\infty)$ such that, for any $x,y,z\in M$, $d_M(x,y)=0$ if and only if $x=y$, $d_M(x,y)=d_M(y,x)$, and $d_M(x,z)\leq d_M(x,y) + d_M(y,z)$.
We write
$ B_M ( \vect{x} , r ) := \{ \vect{y} \in M \mid d_M ( \vect{x} , \vect{y} ) \leq r \} $.
\begin{definition}
  Let $M$ be a metric space and $S$ a discrete subset of $M$.
  We say that 
  \begin{itemize}
    \item $ S $ is a \textit{$\delta$-covering of $M$} if $ \cup_{ \vect{x} \in S } B_M ( \vect{x} , \delta ) = M $,
    \item $ S $ is a \textit{$\delta$-packing of $M$} if 
    $ B_M ( \vect{x} , \delta ) \cap B_M ( \vect{y} , \delta ) = \emptyset $
    for any $ \vect{x} \neq \vect{y} \in S $,
    \item $ S $ is a \textit{$\delta$-net of $M$} if it is both a \textit{$\delta$-covering of $M$} and a \textit{$(\delta/2)$-packing of $M$}.
  \end{itemize}
\end{definition}

Note that, in the poset of $(\delta/2)$-packings ordered under inclusion, a maximal element is a $\delta$-net.
Zorn's lemma shows that, for any metric space $M$, there exists a $\delta$-net.

In the following lemma, 
under some assumptions on $\psi$, we give bounds for the cardinality of a $\delta$-net.
The construction of these bounds is adapted from the proof of the following well known result, see e.g. \cite[Proposition 31.1]{Gruber07}. 
If $C\subset\R^d$ is a convex body with non empty interior such that $C=-C$, then there exists a packing of translated copies of $C$ in $\R^d$ of density at least $2^{-d}$, where, roughly speaking, density means the proportion of $\R^d$ covered by the translated copies of $C$.
%

\begin{lemma}
  \label{lem:Saturated}
  Let $M$ be a space equipped with a measure $ \psi $ and a measurable metric $ d_M $.
  Assume that $\psi(M)<\infty$.
  Let $ \delta_0 > 0 $ and $ S $ be a $\delta$-net of $M$ with $\delta \in ( 0 , \delta_0 ) $.
  Let $ k > 0 $.
  \begin{enumerate}
    \item Assume there exists a constant $ c > 0 $ such that,
    for any $ \vect{x} \in M $ and $ r \in ( 0 , \delta_0 ) $, 
    it holds that
    $ c r^k
    > \psi ( B_M ( \vect{x} , r ) ) $.
    Then
    $ \card{S}
    > c ^{-1} \psi(M) \delta^{-k} $.
    \item Assume there exists a constant $ c' > 0 $ such that,
    for any $ \vect{x} \in M $ and $ r \in ( 0 , \delta_0 ) $,
    it holds that
    $ c' r^k
    < \psi ( B_M ( \vect{x} , r ) ) $.
    Then
    $ \card{S}
    < 2^k c'^{-1} \psi(M) \delta^{-k} $.
  \end{enumerate}
\end{lemma}
\begin{proof}
  To prove $(1)$, we only have to observe that since $ S $ is a $\delta$-covering, we have that
  \[ \psi ( M ) 
  \leq \sum_{ \vect{x} \in S } \psi ( B_M ( \vect{x} , \delta ) )
  < \card{S} c \delta^k \]
  because $ M = \cup_{ \vect{x} \in S } B_M ( \vect{x} , \delta ) $.
  The proof of $(2)$ is similar.
  Since $ S $ is a $(\delta/2)$-packing, we have that
  \[ \psi ( M ) 
  \geq \sum_{ \vect{x} \in S } \psi ( B_M ( \vect{x} , \delta / 2 ) )
  > \card{S} c' \delta^k 2^{-k} \]
  because, for any distinct $ \vect{x} , \vect{y} \in S $, we have 
  $ B_M ( \vect{x} , \delta / 2 ) \cap B_M ( \vect{y} , \delta / 2 ) = \emptyset $.
\end{proof}

In Lemma \ref{lem:deltaNet}, we will apply the previous lemma to the space $M=\partial(K+\B)$, where $K$ is an arbitrary convex body and $M$ is equipped with the surface area measure and the restriction of the euclidean distance.
In this space the balls are caps on the boundary of the convex body $D=K+\B$, where a cap is defined as follow.
For a convex body $D$,
a point $ \vect{d} \in D $ (usually $ \vect{d} \in \partial D $),
and a positive radius $\delta>0$, 
we define the \textit{cap of $D$ of center $\vect{d}$ and radius $\delta$} to be the set
\[ \Cap{D}{\vect{d}}{\delta}
= \{ \vect{y} \in \partial D \mid | \vect{d} - \vect{y} | <\delta\}. \]
Note that our definition differs slightly from the more usual one, where a cap is the intersection of the boundary $\partial D$ with a half-space.
In the next lemma we give bounds for the surface area of caps of radius $\delta\in(0,\delta_0)$, of bodies of the form $K+\B$, with $\delta_0=1$ independent from $K$.
Precise bounds for spherical caps are known, see e.g. Lemma 2.1 in \cite{BriedenAndAl01}, Lemmas 2.2 and 2.3 in \cite{Ball97} or Remark 3.1.8 in \cite{ArtsteinAvidanAndAll15}.
Lemma 6.2 in \cite{RichardsonAndAl08} gives bounds for more general bodies then the sphere, namely those with $\mathcal{C}^2$ boundary of positive curvature, but with a $\delta_0$ depending on $K$.
It does not seems to the author that we can deduce easily Lemma \ref{lem:MeasCap} from these results.
\begin{lemma}
  \label{lem:MeasCap}
  Let $ K \in \mathcal{K} $ and $ D = K + \B $.
  Let $\vect{d} \in \partial D $ and $ \delta \in ( 0 , 1 ) $.
  Then 
  \[ \delta^{d-1} \kappa_{d-1} 2^{-(d-1)}
  < \mathcal{H}^{d-1} ( \Cap{D}{\vect{d}}{\delta} ) 
  < \delta^{d-1} \kappa_{d-1} d .
  \]
\end{lemma}
\begin{proof}
  For the lower bound, we approximate the cap by a $(d-1)$-dimensional disc of radius 
  $ \delta \sqrt{ 1 - \delta^2 / 4 } $
  (see Figure \ref{fig:LowerBoundCap}).
  Let $H$ be the tangent hyperplane to $D$ at $\vect{d}$.
  We have
  \begin{figure}
    \includegraphics*{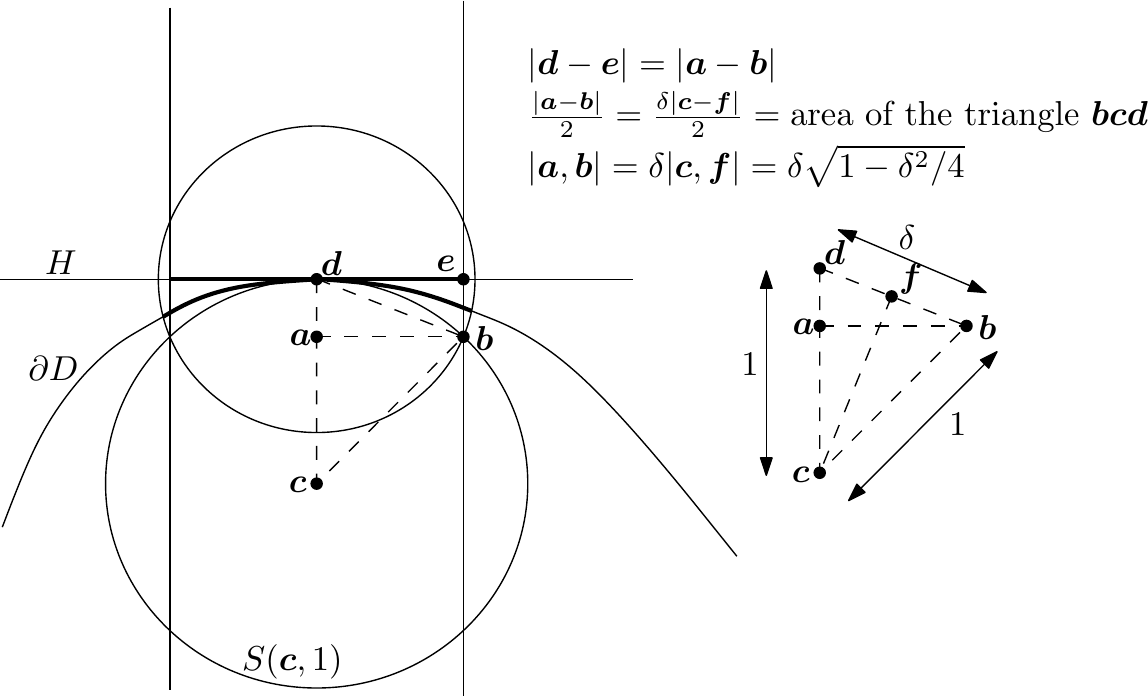}
    \caption{ $\mathcal{H}^{d-1}( \Cap{D}{ \vect{d} }{\delta} )
      \geq \delta^{d-1} \kappa_{d-1} \left( 1 - \frac{ \delta^2 } { 4 } \right)^{(d-1)/2} $}
    \label{fig:LowerBoundCap}
  \end{figure}  
  \begin{align*}
  \mathcal{H}^{d-1}( \Cap{D}{ \vect{d} }{\delta} )
  & \geq 
  \mathcal{H}^{d-1}( H \cap B ( \vect{d} , d ( \vect{d} , \vect{e} ) ) )
  \\
  & = \delta^{d-1} \kappa_{d-1} \left( 1 - \frac{ \delta^2 } { 4 } \right)^{(d-1)/2}
  \\
  & > \delta^{d-1} \kappa_{d-1} \left( \frac{ 3 } { 4 } \right)^{(d-1)/2}
  > \delta^{d-1} \kappa_{d-1} 2^{-(d-1)} .
  \end{align*}
  For the upper bound, we approximate the cap by the union of a $(d-1)$-dimensional disc of radius $\delta$ and the spherical boundary of a cylinder of radius $\delta$ and height $ \delta^2 $ 
  (see Figure \ref{fig:UpperBoundCap}).
  \begin{figure}
    \includegraphics*{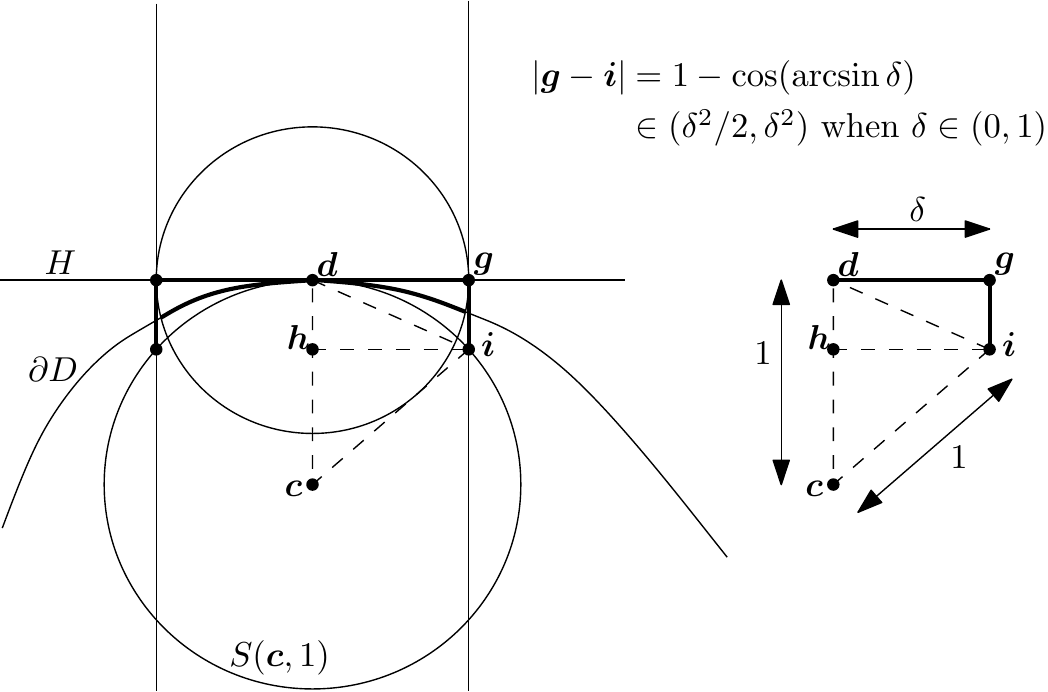}
    \caption{ $\mathcal{H}^{d-1}( \Cap{D}{ \vect{d} }{\delta} )
      < \delta^{d-1} \kappa_{d-1} 
      + \delta^{d-2} \omega_{d-1} \delta^2 $}
    \label{fig:UpperBoundCap}
  \end{figure}  
  Thus
  \begin{align*}
  \mathcal{H}^{d-1}( \Cap{D}{ \vect{d} }{\delta} )
  & < 
  \mathcal{H}^{d-1}(  H \cap B ( \vect{d} , \delta ) )
  + \mathcal{H}^{d-2}(  H \cap S ( \vect{d} , \delta ) ) \delta^2  
  \\
  & =
  \delta^{d-1} \kappa_{d-1} 
  + \delta^{d-2} \omega_{d-1} \delta^2
  \\
  & =
  \delta^{d-1} \kappa_{d-1} \left( 1 + \delta ( d - 1 ) \right)
  <
  \delta^{d-1} \kappa_{d-1} d .
  \end{align*}
\end{proof}

Set 
$ \Cl[global]{12min} 
:= 2 d^{-1} \kappa_{d-1}^{-1} 
= \Approx ( d^{1/2} )^d $
and
$ \Cl[global]{12} 
:= 4^{d} \kappa_{d-1}^{-1}
= \Approx ( d^{1/2} )^d $.
As a direct consequence of the two previous lemmas and the fact that
$\mathcal{H}^{d-1}( \partial D ) = 2 V_{d-1} ( D )$,
we have the following lemma. 
We omit the proof.
\begin{lemma}
  \label{lem:deltaNet}
  Let $ K \in \mathcal{K} $ and $ D = K + \B $,
  $ \delta \in ( 0 , 1 ) $
  and $S$ a $\delta$-net of the boundary $\partial D$. 
  We have that
  \[ \Cr{12min} V_{d-1}(D) \delta^{-(d-1)} 
  < | S | 
  < \Cr{12} V_{d-1}(D) \delta^{-(d-1)} . \]
\end{lemma}

For a convex body $K$ with boundary $\partial K$ of differential class $\mathscr{C}^1$ and $ \vect{x} \in \partial K $, 
we denote by $ \vect{v} ( \vect{x} ) $ the outer unit normal vector of $ K $ at $ \vect{x} $.
Using Lemma \ref{lem:deltaNet}, we can prove the two following lemmas in a similar way as Propositions 2.4 and Proposition 2.7 of \cite{ReisnerSchuttWerner01}.
We will only sketch the proofs.
\begin{lemma}
  \label{lem:deltaNetBis}
  Let $ K \in \mathcal{K} $ with $ \partial K $ of class $ \mathscr{C}^1 $ and 
  $ \delta \in ( 0 , 1 ) $.
  There exists a $\delta$-net of $\partial K$, with respect to the distance 
  $ \MetricBoundary ( \vect{x} , \vect{y} ) 
  = \max ( | \vect{x} - \vect{y} | , | \vect{v} ( \vect{x} ) - \vect{v} ( \vect{y} ) | ) $,
  of cardinality at most 
  $ \Cr{12} V_{d-1}(K+\B ) \,\delta^{-(d-1)} $.
\end{lemma}
\begin{proof}[Sketch of the proof]
  Set $D=K+\B$. 
  Construct a $\delta$-net on the boundary $\partial D$ and then project it onto $\partial K$.
  The bound on the cardinality comes from Lemma \ref{lem:deltaNet}.
\end{proof}

Set 
$ \Cl[global]{12bis}
:= 3^{(d-1)/4} \Cr{12}
= \Approx ( d^{1/2} )^d $.
\begin{lemma}
  \label{prop:bestApproxFixNumberFacets}
  Let $ K \in \mathcal{K} $ and $0<\epsilon<1$.
  Then, there exists a polytope $P_\epsilon\supset K$ with $$\mathrm{d}_H(K,P_\epsilon)<\epsilon$$
  and with number of facets at most
  $$\Cr{12bis} V_{d-1}(K+\B ) \,\epsilon^{-(d-1)/2}.$$
\end{lemma}
\begin{proof}[Sketch of the proof]
  Reduce the proof to the case where $K$ has a smooth boundary. 
  Set an appropriate value $\delta = \delta(\epsilon)$.
  Consider the $\delta$-net $S$ built in Lemma \ref{lem:deltaNetBis}.
  Construct the circumscribed polytope $ P \supset C $ with one facet tangent to $C$ at each point of $S$.
  Finally bound the Hausdorff distance $d_H(C,P)$.
  The bound on the number of facets comes from the bound on the cardinality of the $\delta$-net in Lemma \ref{lem:deltaNetBis}.
\end{proof}
Set 
$ \Cl[global]{13} 
:= \Cr{12bis}^{2/(d-1)} 
= \Approx ( d ) $.
With the last lemma, we can now prove Theorem \ref{thm:intermediate}.
\begin{proof}[Proof of Theorem \ref{thm:intermediate}]
  Let $ n > \Cr{12bis} V_{d-1} ( K + \B ) $.\\
  Set $ \epsilon = \Cr{13}  V_{d-1}(K+\B )^{2/(d-1)} n^{-2/(d-1)}$.
  By the assumption made on $n$, we have $\epsilon<1$.
  Hence, we can apply Lemma~\ref{prop:bestApproxFixNumberFacets}.
  There exists a polytope $P_\epsilon \supset K$ with $d_H(K,P_\epsilon)<\epsilon$ and such that its number of facets is at most
  $$ \Cr{12bis} V_{d-1}(K+\B ) \,\epsilon^{-(d-1)/2}
  = n . $$
  The approximations of the constants $c_i$ using the Landau notation tells us that there exist absolute constants $\Cl[abs]{abs:1}$ and $\Cl[abs]{abs:2}$ such that
  $\Cr{12bis} < \Cr{abs:1} ^d d^{d/2}$ 
  and $ \Cr{13} < \Cr{abs:2} d $
  for any $d$.
  This yields the proof.
\end{proof}

\section{Shape factor}
\label{sec:ShapeFactor}
In this section we define $\ShapeFactor$, a \textit{shape factor}, 
i.e. a scale and translation invariant function on $\mathcal{K}$.
Lemma \ref{lem:propertiesOfgl} tells us how $\ShapeFactor(K)$ describes the elongation of a given convex body $K$.

Set 
$ \Cl[global]{12bisbis} 
:= \Cr{12bis} V_{d-1}(\B )$.
\begin{definition}
  \label{def:fnandgn}
  For any fixed parameter $ l > \Cr{12bisbis} $ we define the functions 
  $ \FunctionBound ,  \ShapeFactorBis ,\ShapeFactor: \mathcal{K} \to (0,\infty) $
  by
  \[ \FunctionBound(K)
  = \sup \{ t \in ( 0 , \infty ) \mid l > \Cr{12bis} V_{d-1}( t K + \B ) \} , \]
  \[  \ShapeFactorBis(K)
  = \inf_{ t \in (0,\FunctionBound (K)) } \frac{ V_{d-1} ( t K + \B )^{ 2/(d-1) } }{t} , \]
  and
  \[ \ShapeFactor(K)
  = \frac{ \ShapeFactorBis(K)}{V_{1}(K)} . \]
\end{definition}
It is clear that the three functions are translation invariant.
One can check that $ \FunctionBound $ is homogeneous of degree $-1$, $ \ShapeFactorBis $ is homogeneous of degree $1$ and $ \ShapeFactor $ is homogeneous of degree $0$.
Therefore, for any fixed $l$, $ \ShapeFactor$ is a shape factor.
The next lemma gives a geometric interpretation of $ \ShapeFactor $.
\begin{lemma}
  \label{lem:propertiesOfgl} \ 
  \begin{enumerate}
    \item For any $K \in \mathcal{K} $, the function $ l \mapsto \ShapeFactor (K) $ is decreasing.
    \item If $d=2$ and $ l > \Cr{12bisbis} $ is fixed, then $ \ShapeFactor $ is constant on $\mathcal{K}$.
    \item If $d \geq 3$, $ l > \Cr{12bisbis} $ is fixed, and $ K \in \mathcal{K} $ is neither an interval nor a ball, then
    \[  \ShapeFactor (I) <  \ShapeFactor (K) <  \ShapeFactor (B) 
    \text{ for any } l > \Cr{12bisbis},\]
    where $I$ denotes an interval and $B$ a ball.
    \item Assume that $ 1 \leq i < j \leq \lceil (d-1)/2 \rceil$.
    There exist constants $ \delta_{i,j} $ and $ n_{i,j} $, both depending on $d$, such that the following holds.
    For any convex body $K\in\mathcal{K}$ and $\epsilon>0$, we have 
    \begin{equation}
    \label{eq:almostFlatBodies}
    \text{if } \frac{ V_j (K)^{1/j} }{ V_i(K)^{1/i} } < \epsilon  \text{ then }  \ShapeFactor[N_{i,j}(\epsilon)] (K) \leq \delta_{i,j} \epsilon^{\beta},
    \end{equation}
    where $ N_{i,j} (\epsilon) := n_{i,j} \epsilon^{-\alpha}$ with $\alpha=2 \lceil (d-1)/2 \rceil (d-1) d^{-1} $,
    and $ \beta = 2 \lceil (d-1)/2 \rceil (d-1)^{-1} d^{-1}$.
  \end{enumerate}
\end{lemma}
\begin{proof}
  (1) is a direct consequence of the definition of $\ShapeFactor$.
  (2) comes from the fact that in this case $V_{d-1}=V_1$ is additive.
  (3) is implied by $\eqref{eq:IneqSurfaceConvexPlusBall}$.
  It only remains to prove~(4).
  
  For the rest of the proof we write $ v_i := V_i (\B)^{1/i} $ for $ i = 1, \ldots , d $.
  Thanks to point $3$ of the present lemma, we have that 
  $\ShapeFactor[N_{i,j}(\epsilon)] (K) 
  \leq \ShapeFactor[N_{i,j}(\epsilon)] (B) $.
  This implies that, without loss of generality, we can assume that $\epsilon < c $, for $c>0$ as small as one need.
  We also reduce the proof to the case $ i=1 $ and $ j=j_0 = \lceil (d-1)/2 \rceil$.
  Because of the isoperimetric inequality $\eqref{eq:IsoperimetricIneq}$, we have
  \begin{equation}
  \label{eq:IV3}
  \frac{ V_{j_0} (K)^{1/j_0} }{ V_1(K) }
  \leq c_{i,j}
  \frac{ V_j (K)^{1/j} }{ V_i(K)^{1/i} } 
  \text{ where }
  c_{i,j} := \frac{v_{j_0} v_i}{ v_j v_1 } .
  \end{equation}
  Assume that there exist constants $\delta_{1,j_0}$ and $n_{1,j_0}$ such that $\eqref{eq:almostFlatBodies}$ holds for $ i=1 $ and $ j=j_0$.
  Let $ 1 \leq i < j \leq j_0 $ and $(i,j)\neq(1,j_0)$.
  We set $ \delta_{i,j} := \delta_{1,j_0} c_{i,j}^\beta $ and $n_{i,j} := n_{1,j_0} c_{i,j}^{-\alpha}$.
  In particular, 
  $ N_{i,j}(\epsilon) 
  = n_{i,j} \epsilon^{-\alpha} 
  = n_{1,j_0} (c_{i,j} \epsilon)^{-\alpha} 
  = N_{1,j_0}(c_{i,j}\epsilon) $.
  Assume that $K$ is such that $V_j(K)^{1/j} V_i(K)^{-1/i} < \epsilon$.
  By $\eqref{eq:IV3}$ we have $V_{j_0}(K)^{1/j_0} V_1(K) < c_{i,j} \epsilon$.
  This implies that 
  $ \ShapeFactor[N_{i,j}(\epsilon)](K) 
  = \ShapeFactor[N_{1,j_0}(c_{i,j}\epsilon)](K) 
  \leq \delta_{1,j_0} (c_{i,j}\epsilon)^\beta
  = \delta_{i,j} \epsilon^\beta $.
  This shows that we only have to consider the case $ i=1 $ and $ j = j_0$.
  
  Since both parts of $\eqref{eq:almostFlatBodies}$ are scale invariant, we also assume without loss of generalities that $V_1(K)=1$.
  Let $ \epsilon \in (0,1) $ and $ l > \Cr{12bisbis} $.
  From now on, we assume that 
  \begin{equation}
  \label{eq:assumtionVj}
  V_{j_0} (K) ^{1/j_0} < \epsilon .
  \end{equation} 
  
  Set 
  \[ p_C  (t) 
  := V_{d-1}( t K + \B ) 
  \overset{\eqref{eq:SteinerType}}{=} \sum_{k=0}^{d-1} \frac{(d-k)\kappa_{d-k}}{2d} V_k(K) t^k .\]
  Observe that it is a strictly increasing and continuous function and that
  \begin{align}
  \label{eq:IV1}
  \ShapeFactor (K) 
  & = \ShapeFactorBis (K)
  = \left( \inf_{ t \in (0,\FunctionBound(K))} t^{-(d-1)/2} p_C (t) \right)^{2/(d-1)} \\
  \notag
  \text{and } \FunctionBound(K)
  & = p_C^{-1} ( \Cr{12bis}^{-1} l  ) . 
  \end{align}
  
  Observe that $ j_0 - 1 - (d-1)/2 \leq -1/2 $.
  Hence, for $ t>1 $, 
  \[ t^{-(d-1)/2} p_C (t) \leq S_1 (K) t^{-1/2} + S_2 (K) t^{(d-1)/2} , \] where
  \[ S_1(K) := \sum_{k=0}^{j_0-1} \frac{(d-k)\kappa_{d-k}}{2d} V_k(K)
  \text{ and }
  S_2(K) := \sum_{k=j_0}^{d-1} \frac{(d-k)\kappa_{d-k}}{2d} V_k(K) .\]
  The isoperimetric inequalities $\eqref{eq:IsoperimetricIneq}$ gives that
  \[ S_1(K) 
  \leq \frac{ \kappa_d }{ 2 } + \sum_{k=1}^{j_0-1} \frac{(d-k)\kappa_{d-k}}{2d}  \left( \frac{v_k}{v_1} \right)^k 
  =: \Cl[global]{IV1} . \]
  It also implies that, for $ k = j_0,\ldots,d-1 $, we have
  $ V_k (K) \leq (v_k / v_{j_0})^k V_{j_0}(K)^{k/j_0} $.
  And since $ V_{j_0}(K)^{k/j_0} < \epsilon^k \leq \epsilon^{j_0} $, it follows that
  \[ S_2(K)
  \leq \sum_{k=j_0}^{d-1} \frac{(d-k)\kappa_{d-k}}{2d} \left( \frac{v_k}{v_{j_0}} \right)^{k} \epsilon^{j_0}
  =:  \Cl[global]{IV2} \epsilon^{j_0} .\]
  Therefore, for $t>1$,
  \begin{equation}
  \label{eq:IV2}
  t^{-(d-1)/2} p_C  (t) 
  \leq \Cr{IV1} t^{-1/2} + \Cr{IV2} \epsilon^{j_0} t^{(d-1)/2} 
  =: q_\epsilon (t) .
  \end{equation}
  Since we want $t^{-(d-1)/2} p_C  (t)$ small, we define $ t_\epsilon > 0 $ such that $q_\epsilon (t_\epsilon)$ is minimal.
  But it holds that the derivative of $q_\epsilon$ is
  \[ q_\epsilon' (t)
  = \frac{ - \Cr{IV1} }2 t^{-3/2} + \frac{ \Cr{IV2} \epsilon^{j_0} (d-1) }2 t^{(d-3)/2}.  \]
  Thus,
  \[ t_\epsilon = \left( \frac{\Cr{IV2} \epsilon^{j_0} (d-1)}{ \Cr{IV1} } \right) ^{-2/d} 
  = \Cr{IV3} \epsilon^{- 2 j_0 / d } \]
  with $ \Cl[global]{IV3} 
  := \left( \Cr{IV2} (d-1) / \Cr{IV1} \right) ^{-2/d}$.
  Now, we observe that 
  \begin{equation*}
  t_\epsilon^{-(d-1)/2} p_C ( t_\epsilon ) 
  \overset{\eqref{eq:IV2}}{\leq} q_\epsilon (t_\epsilon)
  =  \Cr{IV1} (\Cr{IV3} \epsilon^{- 2 j_0 / d })^{-1/2} + \Cr{IV2} \epsilon^{j_0} (\Cr{IV3} \epsilon^{- 2 j_0 / d })^{(d-1)/2}
  = \Cr{IV4} \epsilon^{j_0/d}
  \end{equation*}
  with 
  $\Cl[global]{IV4} := \Cr{IV1} \Cr{IV3}^{-1/2} + \Cr{IV2} \Cr{IV3}^{(d-1)/2}$.
  This implies that if 
  $ \FunctionBound[N_{1,j_0} (\epsilon)] (K) > t_\epsilon $ 
  then
  \[ \ShapeFactor[N_{1,j_0} (\epsilon)] (K)
  \overset{\eqref{eq:IV1}}{\leq} \left( t_\epsilon^{-(d-1)/2} p_C ( t_\epsilon ) \right)^{2/(d-1)}
  \leq \left( \Cr{IV4} \epsilon^{j_0/d} \right)^{ 2 / (d-1) }
  \leq \delta_{1,j_0} \epsilon^\beta
  \]
  with $\delta_{1,j_0} := \Cr{IV4}^{2/(d-1)}$ and 
  $\beta
  :
  = 2 j_0 (d-1)^{-1} d^{-1} $.
  
  It remains only to set $ N_{1,j_0} (\epsilon) $ such that  $ \FunctionBound[ N_{1,j_0} (\epsilon) ] (K) > t_\epsilon $.
  Set 
  \[ \Cl[global]{IV5} := \frac{ \kappa_d }{ 2 } + \sum_{k=1}^{d-1} \frac{(d-k)\kappa_{d-k}}{2d}  \left( \frac{v_k}{v_1} \right)^{k} 
  \text{ and }
  \tilde{p} (t) := \Cr{IV5} t^{d-1}.\]
  Again because of the isoperimetric inequality, we have that $ p_C ( t ) < \tilde{p} (t) $, for any $t>1$.
  Hence if $ u > \tilde{p} (1) = \Cr{IV5} $ then $p_C^{-1} (u) > \tilde{p}^{-1} (u) $.
  Set 
  $$ N_{1,j_0} (\epsilon) 
  := \Cr{12bis} \Cr{IV5} t_{\epsilon}^{d-1}
  =  n_{1,j_0} \epsilon^{- \alpha }  $$ 
  with 
  $ n_{1,j_0} := \Cr{12bis} \Cr{IV5} \Cr{IV3}^{d-1} $
  and
  $ \alpha := 2 j_0 (d-1) d^{-1} $.
  Thus we have
  \[ \FunctionBound[N_{1,j_0} (\epsilon)] (K) 
  = p_C^{-1} ( \Cr{12bis}^{-1}  N_{1,j_0} (\epsilon) )
  = p_C^{-1} ( \Cr{IV5} t_{\epsilon}^{d-1} )
  > \tilde{p}^{-1} ( \Cr{IV5} t_{\epsilon}^{d-1} ) 
  = t_{\epsilon}\]
  whenever $ t_\epsilon > 1 $.
  But $ t_\epsilon > 1 $ when $ \epsilon < \Cr{IV3}^{-1/\alpha} $.
  This completes the proof.
\end{proof}

\section{Proof of Theorem \ref{thm:main}}
\label{sec:MainProof}
Theorem \ref{thm:main} is a direct consequence of the following lemma and point $4$ of Lemma~\ref{lem:propertiesOfgl}.
Let
$ \Cl[global]{13bis} 
> \Cr{13} $.
\begin{lemma}
  Let $ K \in \mathcal{K} $.
  For any $ n >  \Cr{12bisbis} $, we have
  $$ \Cr{boundC} ( K , n ) < \Cr{13bis}  \ShapeFactor[n] (K) V_1(K).$$
  I.e. for any integer $ n >  \Cr{12bisbis} $, there exists a polytope
  $ P \supset K $ with $ n $ facets such that
  $$ \mathrm{d}_H(K,P) 
  < \Cr{13bis}  \ShapeFactor[n] (K) \frac{ V_1(K) }{ n^{2/(d-1)} } . $$
\end{lemma}
\begin{proof}
  The condition 
  $ n >  \Cr{12bisbis} $ 
  implies that 
  $ \FunctionBound[n] (K) $ and $  \ShapeFactor[n] (K) $ are well defined.
  Let $ t \in (0,\FunctionBound[n] (K) ) $.
  We have defined $\FunctionBound[n] (K) $ precisely such that the convex body $ t K $ and the number $n$ satisfy the conditions required to apply Theorem \ref{thm:intermediate}.
  So there exists a polytope $ P_t $ with $ n $ facets such that
  $$ \mathrm{d}_H ( t K , P_t)
  < \Cr{13} V_{d-1} ( tC + \B ) ^{2/(d-1)} n ^{-2/(d-1)} .$$
  Therefore, for any  $ t \in ( 0 , \FunctionBound[n](K) ) $, we see that
  $$ \mathrm{d}_H \left( K , \frac1t P_t \right)
  < \Cr{13} \frac{ V_{d-1} ( t K + \B ) ^{2/(d-1)} }{ t } n^{-2/(d-1)} . $$
  Since $ \Cr{13bis} > \Cr{13} $, there exists $ t_0 \in ( 0 , \FunctionBound[n] ) $ such that
  $$ \mathrm{d}_H \left( K , \frac1{t_0} P_{t_0} \right)
  < \Cr{13bis} \left( \inf_{ t \in ( 0 , \FunctionBound[n] ) } \frac{ V_{d-1} ( t K + \B ) ^{ 2 / (d-1) } }{ t } \right) n^{-2/(d-1)} . $$
  But it holds that
  $$ \inf_{ t \in ( 0 , \FunctionBound[n] ) } \frac{V_{d-1}(t K +\B )^{2/(d-1)}}{t} 
  = \ShapeFactorBis[n] (K) 
  =  \ShapeFactor[n] (K) V_1(K) , $$
  which yields the proof.
\end{proof} 

\bibliographystyle{plain}
\bibliography{Bibliography}	
\end{document}